\documentclass[11pt,a4paper]{article}

\usepackage{fullpage}
\usepackage{amsmath,amssymb,amsfonts,amsthm}
\usepackage{tikz}
\usepackage[english]{isodate}
\usepackage{authblk}
\usepackage{hyperref}

\theoremstyle{plain}
\newtheorem{theorem}{Theorem}
\newtheorem{lemma}{Lemma}
\newtheorem{corollary}{Corollary}

\theoremstyle{definition}

\newtheorem{example}{Example}

\newcommand{\vect}[1]{\mbox{\boldmath$#1$}}

\renewcommand{\geq}{\geqslant}
\renewcommand{\leq}{\leqslant}

\title{Cube-magic labelings of grids}
\author{Rachel Wulan Nirmalasari Wijaya} 
\author{Joe Ryan}
\author{Thomas Kalinowski}
\affil{University of Newcastle, Australia}
\date{\today}

\begin{document}

\maketitle

\begin{abstract}
  We show that the vertices and edges of a $d$-dimensional grid graph $G=(V,E)$ ($d\geqslant 2$) can
  be labeled with the integers from $\{1,\ldots,\lvert V\rvert\}$ and $\{1,\ldots,\lvert E\rvert\}$,
  respectively, in such a way that for every subgraph $H$ isomorphic to a $d$-cube the sum of all
  the labels of $H$ is the same. As a consequence, for every $d\geqslant 2$, every $d$-dimensional
  grid graph is $Q_d$-supermagic where $Q_d$ is the $d$-cube.
\end{abstract}

\section{Introduction}\label{sec:intro}

The graphs considered in this paper are finite, undirected and simple. For a graph $G$, we denote
its vertex set by $V(G)$ and its edge set by $E(G)$. A \emph{graph labeling}, as introduced
in~\cite{Rosa_1966_certainvaluationsvertices}, is an assignment of integers to the vertices or
edges, or both, subject to certain conditions. Over the years, a large variety of different types of
graph labelings have been studied, see~\cite{Gallian_2009_dynamicsurveygraph} for an extensive survey.

For a graph $H$, we say that a graph $G$ admits an \emph{$H$-covering} if every edge of $G$ belongs
to at least one subgraph of $G$ which is isomorphic to $H$. A graph $G=(V,E)$ which admits an
$H$-covering is called \emph{$H$-magic} if there exists a bijection $F: V\cup
E\to\{1,2,\ldots,\lvert V\rvert+\lvert E\rvert\}$ and a constant $c=c(F)$, which we call the
\emph{$H$-magic sum} of $F$, such that 
\[\sum_{v\in V(H')}F(v)+\sum_{e\in E(H')}F(e)=c\]
for every subgraph $H'\subseteq G$ with $H'\cong H$. If in addition $F(V)=\{1,\ldots,\lvert
V\rvert\}$ then we say that the graph $G$ is \emph{$H$-supermagic}. The case where $H$ is a single
edge was studied in~\cite{Enomoto.etal_1998_Superedgemagic}, and the general concept for arbitrary
graphs $H$ was introduced in~\cite{Gutierrez.Llado_2005_Magiccoverings}. Since then $H$-magic and
$H$-supermagic labelings have been studied for a variety of graphs $H$
(\cite{Kojima_2013_supermagiclabelingsCartesian,Llado.Moragas_2007_Cyclemagicgraphs,Maryati.etal_2008_Phsupermagiclabelings,Ngurah.etal_2010_H-supermagic,Wijaya2016}).

In this paper we show that for an integer $d\geqslant 2$, a $d$-dimensional grid graph $G$ is
$Q_d$-supermagic where $Q_d$ denotes the $d$-cube. For $d=2$, the $2$-cube is the same as a 4-cycle
$C_4$ and our result is a consequence of Theorem 1
in~\cite{Kojima_2013_supermagiclabelingsCartesian} which gives sufficient conditions for the
cartesian product of a graph and a path to be $C_4$-supermagic. 

The structure of the paper is as follows. In Section~\ref{sec:results} we fix some notation and
state our main result. Section~\ref{sec:proof} contains the proof which is by induction on the
dimension $d$, where the base case $d=2$ is contained in Section~\ref{sec:base} and the induction
step in Section~\ref{subsec:step}.

\section{Notation and main result}\label{sec:results}

For integers $k\leqslant\ell$ we denote the sets $\{1,2,\ldots,k\}$ and $\{k,k+1,\ldots,\ell\}$ by
$[k]$ and $[k,\ell]$, respectively. For integers $d\geqslant 2$ and $n_1\geq n_2\geq\cdots\geq
n_d\geqslant 2$, let $\textsc{Grid}(d_1,\ldots,n_d)$ denote the $n_1\times\cdots\times n_d$-grid graph, i.e.,
the cartesian product of $d$ paths of lengths $n_1$,\ldots, $n_d$. In other words, the vertex set of
$G(d_1,\ldots,n_d)$ is $V=[n_1]\times\cdots\times[n_d]$ and edge set
\begin{equation*}
E=\left\{\{\vect x,\vect y\}\ :\ \vect x,\,\vect y\in V\text{ and }\sum_{i=1}^d\lvert
x_i-y_i\rvert=1\right\}.
\end{equation*}
The graph $\textsc{Grid}(2,2,\ldots,2)$ is called the $d$-cube and will be denoted by $Q_d$.

To simplify the presentation of our proof we will label the vertices and the edges separately and
then combine the labelings to obtain the $Q_d$-supermagic labeling. A vertex labeling
$f:V\to\{1,2,\ldots,\lvert V\rvert\}$ for a graph $G=(V,E)$ is called $H$-magic if there exists a
constant $c=c(f)$, called the \emph{$H$-magic sum of $f$} such that 
\[\sum_{v\in V(H')}f(v)=c\]
for every subgraph $H'\subseteq G$ with $H'\cong H$. Similarly, an edge labeling
$g:E\to\{1,2,\ldots,\lvert E\rvert\}$ for a graph $G=(V,E)$ is called $H$-magic if there exists a
constant $c'=c'(g)$, called the \emph{$H$-magic sum of $f$} such that 
\[\sum_{e\in E(H')}f(e)=c'\]
for every subgraph $H'\subseteq G$ with $H'\cong H$. An $H$-magic vertex labeling $f$ and
an $H$-magic edge labeling $g$ with $H$-magic sums $c=c(f)$ and $c'=c'(g)$ can be combined to obtain an
$H$-supermagic labeling $F$ with $H$-supermagic sum $c+\lvert E(H)\rvert\lvert V(G)\rvert$ by
setting $F(v)=f(v)$ for all $v\in V(G)$ and $F(e)=g(e)+\lvert V(G)\rvert$ for all $e\in E(G)$.

\begin{theorem}\label{thm:grid_labelings}
Let $d\geqslant 2$ and $n_1\geq n_2\geq\cdots\geq n_d\geqslant 2$ be positive integers, and let
$G=\textsc{Grid}(n_1,\ldots,n_d)$. Then $G$ admits a $Q_d$-magic vertex labeling $f$ and a
$Q_d$-magic edge labeling $g$.
\end{theorem}
Based on the observation about combining $H$-magic vertex and edge labelings we obtain the following corollary.
\begin{corollary}\label{cor:main_result}
 Let $d\geqslant 2$ and $n_1\geq n_2\geq\cdots\geq n_d\geqslant 2$ be positive integers. Then $\textsc{Grid}(n_1,\ldots,n_d)$ is $Q_d$-supermagic.
\end{corollary}

\section{Proof of the main result}\label{sec:proof}
We proceed by induction on $d$. In Section~\ref{sec:base} we treat the base case $d=2$, and present
explicit vertex and edge labelings for grid graphs $\textsc{grid}(n_1,n_2)$. In
Section~\ref{subsec:step} we assume $d\geqslant 3$, and we describe how labelings $f$ and $g$ for
$\textsc{Grid}(n_1,\ldots,n_d)$ can be constructed from the labelings $\tilde f$ and $\tilde g$ for $\textsc{Grid}(n_1,\ldots,n_{d-1})$.
\subsection{The base case $d=2$}\label{sec:base}
In order to describe the labeling in a compact way we use $[P]$ to denote the indicator function for
a statement $P$, i.e.,
\[ [P] =	
\begin{cases}
    1	&	\text{if $P$ is true,}\\
    0	&	\text{if $P$ is false.}\\
  \end{cases}
\]

We define a vertex labeling $f:[n_1]\times[n_2]\to[n_1n_2]$ by
\begin{equation}\label{eq:vertices}
  f(i,j) =	
  \begin{cases}
    (i-1)n_2+j 											& \text{if $i$ is odd and $j$ is odd},\\
    (i-1)n_2+(n_2+1-j) 									& \text{if $i$ is even and $j$ is even},\\
    (n_1-i)n_2+j+[2\mid n_1\wedge2\nmid n_2] 			& \text{if $i$ is odd and $j$ is even},\\
    (n_1-i)n_2+(n_2+1-j+[2\mid n_1\wedge2\nmid n_2]) 	& \text{if $i$ is even and $j$ is odd},
  \end{cases}
\end{equation}
and an edge labeling $g:E\to[2n_1n_2-n_1-n_2]$ by
\begin{align}
  g((i,j),(i,j+1)) &=(i-1)(2n_2-1)+j 		&&\text{for }(i,j)\in[n_1]\times[n_2-1],\label{eq:hedges} \\
  g((i,j),(i+1,j)) &=(n_1-i)(2n_2-1)+1-j &&\text{for }(i,j)\in[n_1-1]\times[n_2].\label{eq:vedges}
\end{align}
\begin{example}\label{ex:2Dsupermagic}
The construction is illustrated in Figure~\ref{fig:5by3grid} for the graph $\textsc{Grid}(5,3)$.
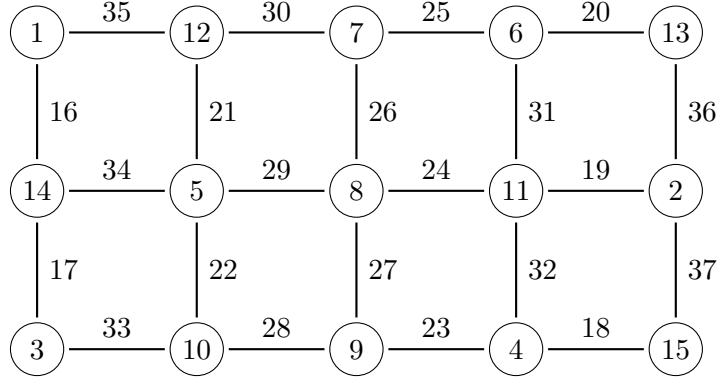
\begin{figure}[htb]
	\centering
	\begin{tikzpicture}[scale=.7,every node/.style={draw,shape=circle,outer sep=2pt,inner sep=1pt,minimum
			size=.7cm}]
		\foreach \j in {1,...,3}{
			\pgfmathtruncatemacro{\jk}{\j*2-1}
			\foreach \k in {1,...,2}{
				\pgfmathtruncatemacro{\kj}{\k*2-1}
				\pgfmathtruncatemacro{\label}{\kj+(\jk-1)*3}
				\node  (\jk\kj) at (3*\jk-3,9-3*\kj) {\label};
			}
			\pgfmathtruncatemacro{\label}{(5-\jk)*3+2}
			\node	(\jk2)	at	(3*\jk-3,3)	{\label};
		}
		
		\foreach \j in {1,...,2}{
			\pgfmathtruncatemacro{\jk}{\j*2}
			\foreach \k in {1,...,2}{
				\pgfmathtruncatemacro{\kj}{\k*2-1}
				\pgfmathtruncatemacro{\label}{(5-\jk)*3+4-\kj}
				\node  (\jk\kj) at (3*\jk-3,9-3*\kj) {\label};
			}
			\pgfmathtruncatemacro{\label}{\jk*3-1}
			\node	(\jk2)	at	(3*\jk-3,3)	{\label};
		}
		
		\foreach \k in {1,...,2}{
			\pgfmathtruncatemacro{\next}{1+\k}
			\foreach \j in {1,...,5}{
				\pgfmathtruncatemacro{\label}{(\j-1)*5+\k+15}
				\draw[thick] (\j\k) to node[draw=none,right=-2pt] {\label} 	(\j\next);
			}
		}
		
		\foreach \k in {1,...,3}{
			\foreach \j in {1,...,4}{
				\pgfmathtruncatemacro{\jnext}{\j+1}
				\pgfmathtruncatemacro{\label}{(5-\j)*5+16-\k}
				\draw[thick] (\j\k) to node[draw=none,above=-4pt,sloped] {\label} 	(\jnext\k);
			}
		}
	\end{tikzpicture}
	\caption{A $C_4$-supermagic labeling for $\textsc{Grid}(5,3)$. The vertex labels $f(i,j)$
          are given by~\eqref{eq:vertices}, and the edge labels are $g(e)+15$ where $g$ is defined
          by~\eqref{eq:hedges} and~\eqref{eq:vedges}.}\label{fig:5by3grid}
\end{figure}
  \end{example}

\begin{lemma}\label{lem:Vsum2d}
The function $f$ defined by~\eqref{eq:vertices} is a $Q_2$-magic vertex labeling for $\textsc{Grid}(n_1,n_2)$ with $Q_2$-magic sum 
\[c(f)=
\begin{cases}
  2(n_1n_2+1) &\text{if $n_1$ is odd or $n_2$ is even},\\
  2(n_1n_2+2) &\text{if $n_1$ is even and $n_2$ is odd}.
\end{cases}
\]
\end{lemma}
\begin{proof}
For every $(i,j)\in[n_1-1]\times[n_2-1]$ we have
\begin{align*}
f(i,j)&+f(i,j+1)+f(i+1,j)+f(i+1,j+1)\\
&=(i-1)n_2+j+(n_1-i)n_2+(j+1)+[2\mid n_1\wedge2\nmid n_2]+(n_1-(i+1))n_2\\
  &\qquad	+n_2+1-(j+1)+[2\mid n_1\wedge2\nmid n_2]+((i+1)-1)n_2+n_2+1-(j+1)\\
 &= (2n_1-2)n_2+2n+2+2[2\mid n_1\wedge2\nmid n_2]\\
&= 2(n_1n_2+1+[2\mid n_1\wedge2\nmid n_2]).\qedhere
\end{align*}
\end{proof}

\begin{lemma}\label{lem:Esum2d}
  The function $g$ defined by~\eqref{eq:hedges} and~\eqref{eq:vedges} is a $Q_2$-magic edge labeling for $\textsc{Grid}(n_1,n_2)$ with $Q_2$-magic sum 
$c(g)=(2n_1-1)(2n_2-1)+1$.
\end{lemma}
\begin{proof}
For every $(i,j)\in[n_1-1]\times[n_2-1]$, we have
  \begin{align*}
    g((i,j)&,(i,j+1))  +g((i,j),(i+1,j))+g((i+1,j),(i+1,j+1))+g((i,j+1),(i+1,j+1))\\
   &=	(i-1)(2n_2-1)+j+(n_1-i)(2n_2-1)+1-j+((i+1)-1)(2n_2-1)+j\\
   &\quad	+(n_1-i)(2n_2-1)+1-(j+1)\\
   &= (2n_1-1)(2n_2-1)+1.\qedhere
  \end{align*}
\end{proof}

Combining these two labelings as described in Section~\ref{sec:results} we obtain a $Q_2$-supermagic
labeling $F$ with $Q_2$-supermagic sum
\[c(F)= 10n_1n_2-2n_1-2n_2 +
\begin{cases}
  4 &\text{if $n_1$ is odd or $n_2$ is even},\\
  6 &\text{if $n_1$ is even and $n_2$ is odd}.
\end{cases}
\]

\subsection{The induction step}\label{subsec:step}

We now assume $d\geqslant 3$. By induction, there exist $Q_{d-1}$-magic labelings $\tilde f$ and
$\tilde g$ with $Q_{d-1}$-magic sums $S=c(f)$ and $S'=c'(g)$ for $\textsc{Grid}(n_1,\ldots,n_{d-1})$. We define
$f:[n_1]\times\cdots\times[n_d]\to\left\{1,\ldots,\prod^{d}_{i=1}n_i\right\}$ by
\begin{equation}\label{eq:Dvertexlabel}
  f(x_1,\ldots,x_d)=	
\begin{cases}
  \tilde f(x_1,\ldots,x_{d-1})+(x_d-1)\prod\limits^{d-1}_{i=1}n_i & \text{if }x_1+\cdots+x_{d-1}\text{ is
    even},\\
  \tilde f(x_1,\ldots,x_{d-1})+(n_d-x_d)\prod\limits^{d-1}_{i=1}n_i & \text{if }x_1+\cdots+x_{d-1}\text{ is
    odd}.
\end{cases}
\end{equation}
\begin{example}
The vertex labeling~\eqref{eq:Dvertexlabel} is illustrated in Figure~\ref{fig:533grid_vertex} for
the graph $\textsc{Grid}(5,3,3)$ where $\tilde f$ is the vertex labeling for $\textsc{Grid}(5,3)$
presented in Example~\ref{ex:2Dsupermagic}. 
\begin{figure}[htb]
	\centering
	\begin{tikzpicture}[scale=1,every node/.style={draw,shape=circle,outer sep=2pt,inner sep=1pt,minimum
		size=.6cm}]
		\foreach \k in {1,...,3}{
			\foreach \j in {1,...,3}{
				\pgfmathtruncatemacro{\jk}{\j*2-1}
				\pgfmathtruncatemacro{\label}{(\k-1)*15+(\jk-1)*3+1}
				\node  (\jk1\k) at (3*\jk-3,9-3*\k) {\label};
				\pgfmathtruncatemacro{\label}{(3-\k)*15+(5-\jk)*3+2}
				\node  (\jk2\k) at (3*\jk-1.1,10.15-3*\k) {\label};
				\pgfmathtruncatemacro{\label}{(\k-1)*15+(\jk-1)*3+3}
				\node  (\jk3\k) at (3*\jk+0.8,11.3-3*\k) {\label};
			}
			
			\foreach \j in {1,...,2}{
				\pgfmathtruncatemacro{\jk}{\j*2}
				\pgfmathtruncatemacro{\label}{(3-\k)*15+(5-\jk)*3+3}
				\node  (\jk1\k) at (3*\jk-3,9-3*\k) {\label};
				\pgfmathtruncatemacro{\label}{(\k-1)*15+(\jk-1)*3+2}
				\node  (\jk2\k) at (3*\jk-1.1,10.15-3*\k) {\label};
				\pgfmathtruncatemacro{\label}{(3-\k)*15+(5-\jk)*3+1}
				\node  (\jk3\k) at (3*\jk+0.8,11.3-3*\k) {\label};
			}
		}
		
		\foreach \k in {1,...,2}{
			\pgfmathtruncatemacro{\next}{1+\k}
			\foreach \j in {1,...,5}{
				\draw[thick] (\j1\k) to node[draw=none,above=5pt,right=-6pt,sloped] {} 	(\j1\next);
				\draw[thick] (\j2\k) to node[draw=none,below=5pt,right=-10pt,sloped] {} 	(\j2\next);
				\draw[thick] (\j3\k) to node[draw=none,below=5pt,left=-5pt,sloped] {} 	(\j3\next);
			}
		}
		
		\foreach \k in {1,...,3}{
			\foreach \j in {1,...,4}{
				\pgfmathtruncatemacro{\jnext}{\j+1}
				\draw[thick] (\j1\k) to node[draw=none,above=4pt,right=10pt,sloped] {} 	(\jnext1\k);
				\draw[thick] (\j2\k) to node[draw=none,above=4pt,right=8pt,sloped] {} 	(\jnext2\k);
				\draw[thick] (\j3\k) to node[draw=none,above=4pt,right=1pt,sloped] {} 	(\jnext3\k);
			}
			\foreach \j in {1,...,5}{
				\draw[thick] (\j1\k) to node[draw=none,below=-4pt,sloped] {} 	(\j2\k);
				\draw[thick] (\j2\k) to node[draw=none,above=4pt,left=-3pt,sloped] {} 	(\j3\k);
			}
		}	
	\end{tikzpicture}
	\caption{A $Q_3$-magic vertex labeling with magic sum $c(f)=184$ for $\textsc{Grid}(5,3,3)$.}\label{fig:533grid_vertex}
\end{figure}
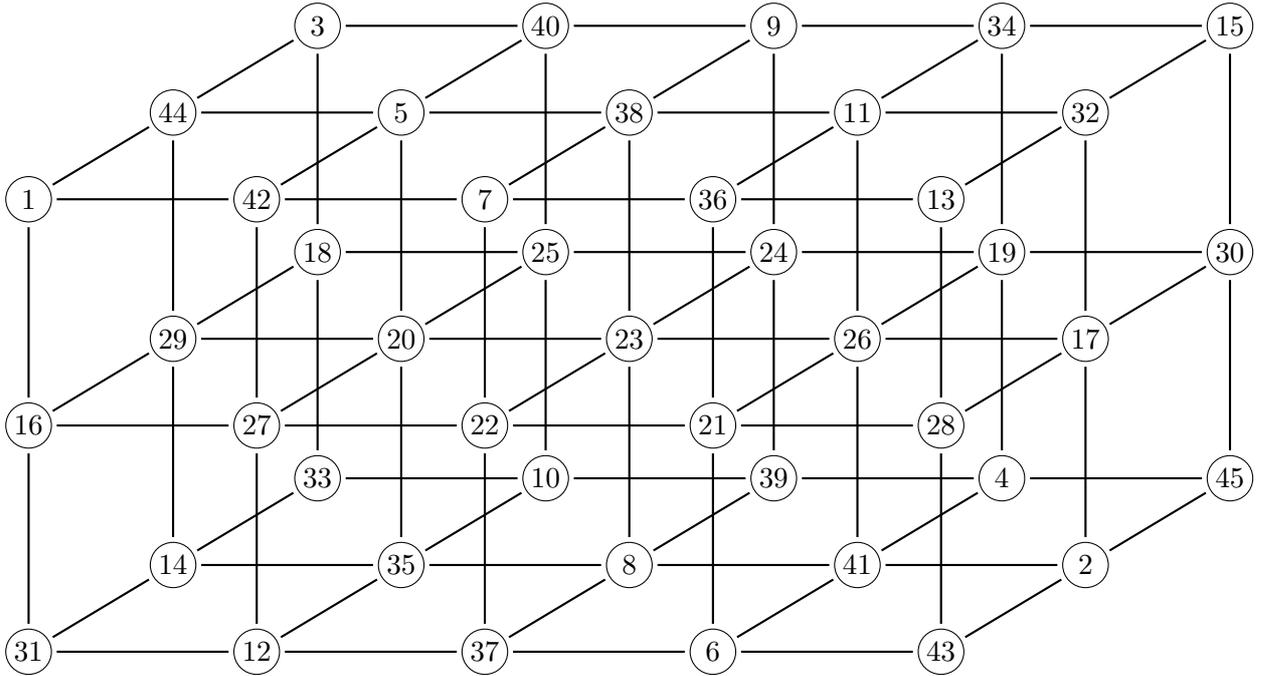
\end{example}

\begin{lemma}\label{lem:vertex_induction}
The function $f$ defined by~(\ref{eq:Dvertexlabel}) a $Q_d$-magic vertex labeling with $Q_d$-magic
sum 
\[c(f)=2S+2^{d-1}(n_d-1)\prod^{d-1}_{i=1}n_i.\]
\end{lemma}
Before proving the lemma we illustrate the basic idea using the $Q_3$-subgraph of
$\textsc{Grid}(5,3,3)$ shown in Figure~\ref{fig:subcube_533}.
\begin{figure}[htb]
  \centering
  \begin{tikzpicture}[xscale=1.8,yscale=1.2,every node/.style={draw,shape=circle,outer sep=2pt,inner sep=1pt,minimum
		size=.6cm}]
  \node[label={west:{\small$(2,1,2)$}}]  (000) at (0,0) {27};
  \node[label={west,yshift=.2cm:{\small$(3,1,2)$}}]  (001) at (2,0) {22};
  \node[label={west,yshift=.2cm:{\small$(3,1,1)$}}]  (011) at (2,2) {7};
  \node[label={west:{\small$(2,1,1)$}}]  (010) at (0,2) {42};
  \node[label={east,yshift=.2cm:{\small$(2,2,2)$}}]  (100) at (.8,.7) {20};
  \node[label={east:{\small$(3,2,2)$}}]  (101) at (2.8,.7) {23};
  \node[label={east:{\small$(3,2,1)$}}]  (111) at (2.8,2.7) {38};
  \node[label={east,yshift=.2cm:{\small$(2,2,1)$}}]  (110) at (0.8,2.7) {5};
  \draw[thick] (000) -- (001) -- (011) -- (010) -- (000) -- (100) -- (101) -- (111) -- (110) --
  (100);
  \draw[thick] (010) -- (110);
  \draw[thick] (011) -- (111);
  \draw[thick] (001) -- (101);
  \end{tikzpicture}
  \caption{A vertex-labeled $Q_3$-subgraph of $\textsc{Grid}(5,3,3)$. The triples next to the
    vertices indicate their position in the grid.}\label{fig:subcube_533}
\end{figure}
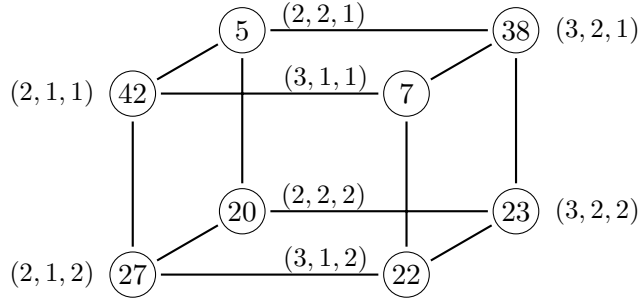
This subgraph corresponds to the square labeled 12--7--8--5 in Figure~\ref{fig:5by3grid}, and
by~\eqref{eq:Dvertexlabel} the vertex labels for the top and the bottom square of the cube in
Figure~\ref{fig:subcube_533} are
\begin{align*}
27&=12+15,& 22&=7+15,& 23&=8+15,& 20&=5+15\\
42&=12+2\times 15,& 7&=7,& 38&=8+2\times 15,& 5&=5
\end{align*}
and therefore the sum of the vertex labels is $2(12+7+8+5+4\times 15)=2S+8\times 15$, where $S$ is
the magic sum of the vertex labeling in Figure~\ref{fig:5by3grid}. 
\begin{proof}[Proof of Lemma~\ref{lem:vertex_induction}] 
Fix $i\in[n_d]$ and $(x_1,\ldots,x_{d-1})\in[n_1-1]\times\cdots\times[n_{d-1}-1]$. Let $H\cong Q_{d-1}$ be the subgraph of $\textsc{Grid}(n_1,\ldots,n_d)$
induced by 
\[V(H)=\left\{(x_1+\varepsilon_1,\ldots,x_{d-1}+\varepsilon_{d-1},i)\ :\ (\varepsilon_1,\ldots,\varepsilon_{d-1})\in\{0,1\}^{d-1}\right\}.\] 
Using the fact that exactly half of the $2^{d-1}$ vertices of $H$ have even coordinate sum, we
obtain from~(\ref{eq:Dvertexlabel}),
\begin{multline}\label{eq:face}
  \sum_{v\in V(H)} f(v)	= \sum_{\vect\varepsilon\in\{0,1\}^{d-1}} \tilde
  f(x_1+\varepsilon_1,\ldots,x_{d-1}+\varepsilon_{d-1})+2^{d-2}(i-1)\prod^{d-1}_{i=1}n_i+2^{d-2}(n_d-i)\prod^{d-1}_{i=1}n_i\\
  = S+2^{d-2}\prod^{d-1}_{i=1}n_i(i-1+n_d-i)= S+2^{d-2}(n_d-1)\prod^{d-1}_{i=1}n_i.
\end{multline}
A subgraph $H$ of $\textsc{Grid}(n_1,\ldots,n_d)$ is isomorphic to $Q_d$ if and only if its vertex set is 
\[V(H)=\{(x_1+\varepsilon_1,\ldots,x_{d-1}+\varepsilon_{d-1},x_d+\varepsilon_d)\ :\ 
(\varepsilon_1,\ldots,\varepsilon_{d-1},\varepsilon_d)\in\{0,1\}^{d}\}\]
for some $(x_1,\ldots,x_{d})\in[n_1-1]\times\cdots\times[n_{d}-1]$. Using~(\ref{eq:face}), this implies 
\[\sum_{v\in V(H)} f(v)	=2\left(S+2^{d-2}(n_d-1)\prod^{d-1}_{i=1}n_i\right).\qedhere\]
\end{proof}
We define an edge labeling $g:E\to\{1,\ldots,\lvert E\rvert\}$ for $G=\textsc{Grid}(n_1,\ldots,n_d)$
as follows. Let $N=\lvert V(\textsc{Grid}(n_1,\ldots,n_{d-1}))\rvert$ and $M=\lvert
E(\textsc{Grid}(n_1,\ldots,n_{d-1}))\rvert$. For an edge $e=\{\vect x,\vect y\}$ with $y_d=x_d+1$, we set
\begin{equation}\label{eq:ConnectingEdges}
  g(e)=	\tilde f(x_1,\ldots,x_{d-1})+n_dM+ \begin{cases}
    (x_d-1)N & \text{ if }x_1+\cdots+x_{d-1}\text{ is odd},\\
    (n_d-1-x_d)N & \text{ if }x_1+\cdots+x_{d-1}\text{ is even}.
  \end{cases}
\end{equation}
For the remaining edges we distinguish two cases.
\begin{description}
\item[Case 1.] $d$ is odd. For an edge $e=\{\vect x,\vect y\}$ with $y_i=x_i+1$, $i\leq d-1$, we set
  \begin{equation}\label{eq:EdgeOddCase}
    g(e)=	
      \tilde g(\{(x_1,\ldots,x_{d-1}),(y_1,\ldots,y_{d-1})\}) + 
\begin{cases}(x_d-1)M & \text{ if }i\text{ is odd},\\
  (n_d-x_d)M & \text{ if }i\text{ is even}.
    \end{cases}
  \end{equation}
\item[Case 2.] $d$ is even. For an edge $e=\{\vect x,\vect y\}$ with $y_i=x_i+1$, $i\leq d-2$, we
  set
  \begin{equation}\label{eq:EdgeEvenCase1}
    g(e)=	
      \tilde g(\{(x_1,\ldots,x_{d-1}),(y_1,\ldots,y_{d-1})\}) + 
\begin{cases}(x_d-1)M & \text{ if }i\text{ is odd},\\
    (n_d-x_d)M & \text{ if }i\text{ is even}.
\end{cases}
  \end{equation}	
  For an edge $e=\{\vect x,\vect y\}$ with $y_{d-1}=x_{d-1}+1$, we set
  \begin{equation}\label{eq:EdgeEvenCase2}
    g(e)=	
      \tilde g(\{(x_1,\ldots,x_{d-1}),(y_1,\ldots,y_{d-1})\}) + 
     \begin{cases} (x_d-1)M & \text{ if }\sum_{i=1}^{d-2}x_i\text{ is odd},\\
      (n_d-x_d)M & \text{ if }\sum_{i=1}^{d-2}x_i\text{ is even}.
    \end{cases}
  \end{equation}
\end{description}

\begin{example}\label{ex:3Dedgelabels}
The edge labeling given by~(\ref{eq:ConnectingEdges}) to~(\ref{eq:EdgeEvenCase2}) is illustrated in Figure~\ref{fig:533grid_edge} for
the graph $\textsc{Grid}(5,3,3)$. The underlying labelings for $\tilde f$ and $\tilde g$ for
$\textsc{Grid}(5,3)$ are the labelings from Example~\ref{ex:2Dsupermagic}. 
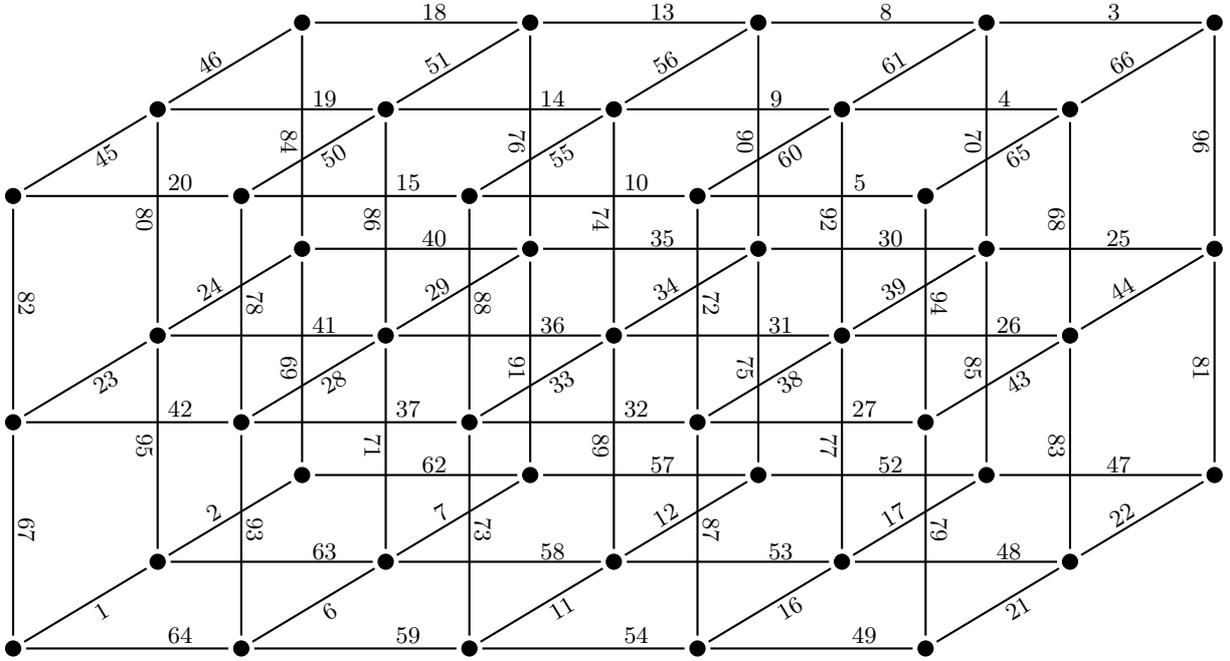
\begin{figure}[htb]
  \centering
  \begin{tikzpicture}[scale=1,every node/.style={draw,shape=circle,outer sep=2pt,inner
      sep=1pt,minimum size=.2cm}]
    \footnotesize 
\foreach \k in {1,...,3}{ 
\foreach \j in {1,...,5}{ 
\node[fill=black] (\j1\k) at (3*\j-3,9-3*\k) {}; 
\node[fill=black] (\j2\k) at (3*\j-1.1,10.15-3*\k) {}; 
\node[fill=black] (\j3\k) at (3*\j+0.8,11.3-3*\k) {}; } } 

\foreach \k in {1,...,2}{ 
\pgfmathtruncatemacro{\next}{1+\k} 
\foreach \j in {1,...,3}{
    \pgfmathtruncatemacro{\jk}{\j*2-1} 
    \pgfmathtruncatemacro{\label}{66+(2-\k)*15+(\jk-1)*3+1}
    \draw[thick] (\jk1\k) to node[draw=none,above=5pt,right=-11pt,sloped] {\label} (\jk1\next);
    \pgfmathtruncatemacro{\label}{66+(\k-1)*15+(5-\jk)*3+2} \draw[thick] (\jk2\k) to
    node[draw=none,below=5pt,right=-10pt,sloped] {\label} (\jk2\next);
    \pgfmathtruncatemacro{\label}{66+(2-\k)*15+(\jk-1)*3+3} \draw[thick] (\jk3\k) to
    node[draw=none,below=5pt,left=-11pt,sloped] {\label} (\jk3\next); }
			
   \foreach \j in {1,...,2}{ 
 \pgfmathtruncatemacro{\jk}{\j*2}
 \pgfmathtruncatemacro{\label}{66+(\k-1)*15+(5-\jk)*3+3} 
 \draw[thick] (\jk1\k) to node[draw=none,above=5pt,right=-11pt,sloped] {\label} (\jk1\next);
 \pgfmathtruncatemacro{\label}{66+(2-\k)*15+(\jk-1)*3+2} 
 \draw[thick] (\jk2\k) to node[draw=none,below=5pt,right=-10pt,sloped] {\label} (\jk2\next);
 \pgfmathtruncatemacro{\label}{66+(\k-1)*15+(5-\jk)*3+1} 
 \draw[thick] (\jk3\k) to node[draw=none,below=5pt,left=-11pt,sloped] {\label} (\jk3\next); } 
}
		
		\foreach \k in {1,...,3}{
		 	\foreach \j in {1,...,4}{
		 		\pgfmathtruncatemacro{\jnext}{\j+1}
		 		\pgfmathtruncatemacro{\label}{(\k-1)*22+(4-\j)*5+5}
		 		\draw[thick] (\j1\k) to node[draw=none,above=5pt,right=11pt,sloped] {\label} 	(\jnext1\k);
		 		\pgfmathtruncatemacro{\label}{(\k-1)*22+(4-\j)*5+4}
		 		\draw[thick] (\j2\k) to node[draw=none,above=4pt,right=11pt,sloped] {\label} 	(\jnext2\k);
		 		\pgfmathtruncatemacro{\label}{(\k-1)*22+(4-\j)*5+3}
		 		\draw[thick] (\j3\k) to node[draw=none,above=4pt,right=-2pt,sloped] {\label} 	(\jnext3\k);
		 	}
			\foreach \j in {1,...,5}{
				\pgfmathtruncatemacro{\label}{(3-\k)*22+(\j-1)*5+1}
				\draw[thick] (\j1\k) to node[draw=none,below=5pt,right=-3pt,sloped] {\label} 	(\j2\k);
				\pgfmathtruncatemacro{\label}{(3-\k)*22+(\j-1)*5+2}
				\draw[thick] (\j2\k) to node[draw=none,above=5pt,left=-3pt,sloped] {\label} 	(\j3\k);
			}
		 }
	\end{tikzpicture}
	\caption{A $Q_3$-magic edge labeling with magic sum $594$ for $\textsc{Grid}(5,3,3)$.}\label{fig:533grid_edge}
\end{figure}
\end{example}

\begin{lemma}\label{lem:edge_induction}
 The function $g$ defined by~(\ref{eq:ConnectingEdges}) to~(\ref{eq:EdgeEvenCase2}) is a $Q_d$-magic edge labeling with $Q_d$-magic sum
\[c'(g)=S+2S'+2^{d-2}(n_d-2)N+2^{d-2}\left(2n_d+(d-1)(n_d-1)\right)M.\]
\end{lemma}
Again, we illustrate the basic idea before going into the formal
proof. Figure~\ref{fig:subcube_533_edge} shows the edge labels for the same subgraph of
$\textsc{Grid}(5,3,3)$ as in the illustration for Lemma~\ref{lem:vertex_induction} (see
Figure~\ref{fig:subcube_533}).  
\begin{figure}[htb]
  \begin{minipage}[b]{.49\linewidth}
  \centering
  \begin{tikzpicture}[xscale=1.8,yscale=1.6,every node/.style={draw,shape=circle,outer sep=2pt,inner sep=1pt,minimum
		size=.2cm}]
  \node[fill=black,label={west:{\small$(2,1,2)$}}]  (000) at (0,0) {};
  \node[fill=black,label={west,yshift=.2cm:{\small$(3,1,2)$}}]  (001) at (2,0) {};
  \node[fill=black,label={west,yshift=.2cm:{\small$(3,1,1)$}}]  (011) at (2,2) {};
  \node[fill=black,label={west:{\small$(2,1,1)$}}]  (010) at (0,2) {};
  \node[fill=black,label={east,yshift=.2cm:{\small$(2,2,2)$}}]  (100) at (.8,.7) {};
  \node[fill=black,label={east:{\small$(3,2,2)$}}]  (101) at (2.8,.7) {};
  \node[fill=black,label={north,yshift=-.5cm:{\small$(3,2,1)$}}]  (111) at (2.8,2.7) {};
  \node[fill=black,label={north,yshift=-.5cm:{\small$(2,2,1)$}}]  (110) at (0.8,2.7) {};
  \draw[thick] (000) to node[draw=none,above=5pt,left=-3pt] {37} (001);
  \draw[thick] (000) to node[draw=none,above=5pt,left=-3pt] {78} (010);
  \draw[thick] (000) to node[draw=none,above=5pt,left=-3pt] {28} (100);
  \draw[thick] (001) to node[draw=none,above=5pt,left=-3pt] {88} (011);
  \draw[thick] (001) to node[draw=none,below] {33} (101);
  \draw[thick] (010) to node[draw=none,below,near end] {15} (011);
  \draw[thick] (010) to node[draw=none,above=5pt,left=-3pt] {50} (110);
  \draw[thick] (111) to node[draw=none,above] {14} (110);
  \draw[thick] (111) to node[draw=none,right] {74} (101);
  \draw[thick] (111) to node[draw=none,above=5pt,left=-3pt] {55} (011);
  \draw[thick] (100) to node[draw=none,above,very near end] {36} (101);
  \draw[thick] (100) to node[draw=none,left, near start] {86} (110);
  \end{tikzpicture}
  \caption{An edge-labeled $Q_3$-subgraph of $\textsc{Grid}(5,3,3)$.}\label{fig:subcube_533_edge}  
  \end{minipage}\hfill
  \begin{minipage}[b]{.49\linewidth}
\centering
\begin{tikzpicture}[scale=.6,every node/.style={draw,shape=circle,outer sep=2pt,inner sep=1pt,minimum
			size=.6cm}]
{\footnotesize
		\foreach \j in {1,...,3}{
			\pgfmathtruncatemacro{\jk}{\j*2-1}
			\foreach \k in {1,...,2}{
				\pgfmathtruncatemacro{\kj}{\k*2-1}
				\pgfmathtruncatemacro{\label}{\kj+(\jk-1)*3}
				\node  (\jk\kj) at (3*\jk-3,9-3*\kj) {\label};
			}
			\pgfmathtruncatemacro{\label}{(5-\jk)*3+2}
			\node	(\jk2)	at	(3*\jk-3,3)	{\label};
		}
		
		\foreach \j in {1,...,2}{
			\pgfmathtruncatemacro{\jk}{\j*2}
			\foreach \k in {1,...,2}{
				\pgfmathtruncatemacro{\kj}{\k*2-1}
				\pgfmathtruncatemacro{\label}{(5-\jk)*3+4-\kj}
				\node  (\jk\kj) at (3*\jk-3,9-3*\kj) {\label};
			}
			\pgfmathtruncatemacro{\label}{\jk*3-1}
			\node	(\jk2)	at	(3*\jk-3,3)	{\label};
		}
		
		\foreach \k in {1,...,2}{
			\pgfmathtruncatemacro{\next}{1+\k}
			\foreach \j in {1,...,5}{
				\pgfmathtruncatemacro{\label}{(\j-1)*5+\k}
				\draw[thick] (\j\k) to node[draw=none,right=-2pt] {\label} 	(\j\next);
			}
		}
		
		\foreach \k in {1,...,3}{
			\foreach \j in {1,...,4}{
				\pgfmathtruncatemacro{\jnext}{\j+1}
				\pgfmathtruncatemacro{\label}{(5-\j)*5+1-\k}
				\draw[thick] (\j\k) to node[draw=none,above=-4pt,sloped] {\label} 	(\jnext\k);
			}
		}
}
	\end{tikzpicture}
	\caption{$Q_2$-vertex magic and $Q_2$-edge magic labeling of $\textsc{Grid}(5,3)$.}\label{fig:5by3grid_alt}    
  \end{minipage}
\end{figure}
The top and bottom squares of the cube in Figure~\ref{fig:subcube_533_edge} correspond to the square
labeled 12--7--8--5 in Figure~\ref{fig:5by3grid_alt}. In this example we have $N=15$ and $M=22$, and according to~\eqref{eq:ConnectingEdges}, the
labels of the vertical edges are
\begin{align*}
  78 &= 12+3\times 22, & 88 &= 7+3\times 22+15 , & 74 &= 8+3\times 22, & 86 &= 5+3\times 22+15,
\end{align*}
so the sum of the labels of the vertical edges is $S+12\times 22 +2\times 15$, where $S=12+7+8+5$ is
the vertex-magic sum in Figure~\ref{fig:5by3grid_alt}. According to~\eqref{eq:EdgeOddCase}, the
labels of the edges in the top and the bottom square are
\begin{align*}
  15 &= 15, & 55 &= 11+2\times 22 , & 14 &= 14, & 50 &= 6+4\times 22,\\
  37 &= 15+22, & 33 &= 11+22 , & 36 &= 14+22, & 28 &= 6+22,
\end{align*}
and therefore the sum of these edge labels is $2(15+11+14+6+4\times 22)=2S'+8\times 22$, where $S'=15+11+14+6$ is
the edge-magic sum in Figure~\ref{fig:5by3grid_alt}. Adding all edge labels we obtain $S+2S'+2\times
15+20\times 22$ as claimed in the lemma.
\begin{proof}[Proof of Lemma~\ref{lem:edge_induction}] 
Fix a subgraph $H\subseteq\textsc{Grid}(n_1,\ldots,n_d)$ with $H\cong Q_d$. Its vertex set is
\[V(H)=\{(x_1+\varepsilon_1,\ldots,x_{d-1}+\varepsilon_{d-1},x_d+\varepsilon_d)\ :\ 
(\varepsilon_1,\ldots,\varepsilon_d)\in\{0,1\}^{d}\}\]
for some $(x_1,\ldots,x_{d})\in[n_1-1]\times\cdots\times[n_{d}-1]$. We partition the vertex set as
$V(H)=V_0(H)\cup V_1(H)$ and the edge set as $E(H)=E_0(H)\cup E_1(H)\cup E_2(H)$ where
    \begin{align*}
     V_\varepsilon(H) &= \left\{\vect y\in V(H)\ :\ y_d=x_d+\varepsilon\right\}&&\text{for }\varepsilon\in\{0,1\},\\
     E_\varepsilon(H) &= \left\{\{\vect y,\vect z\}\in E(H)\ :\ y_d=z_d=x_d+\varepsilon\right\}&&\text{for }\varepsilon\in\{0,1\},\\
     E_2(H)&=  \left\{\{\vect y,\vect z\}\in E(H)\ :\ y_d=x_d,\,z_d=x_d+1\right\}.
    \end{align*}
    Note that $\lvert V_0(H)\rvert=\lvert V_1(H)\rvert=\lvert E_2(H)\rvert=\lvert
    V(Q_{d-1})\rvert=2^{d-1}$ and $\lvert E_0(H)\rvert=\lvert E_1(H)\rvert=\lvert
    E(Q_{d-1})\rvert=(d-1)2^{d-2}$. Using the fact that $y_1+\cdots+y_{d-1}$ is even for exactly
    half of the edges $\{\vect y,\,\vect z\}\in E_2(H)$ and that the subgraph of $G$ induced by
    $V_0(H)$ is isomorphic to $Q_{d-1}$, we obtain from~(\ref{eq:ConnectingEdges}),
\begin{multline}\label{eq:connectingedgesproof}
  \sum_{e\in E_2(H)}g(e)=\sum_{\vect y\in V_0(H)}\tilde
  f(y_1,\ldots,y_{d-1})+2^{d-1}n_dM+2^{d-2}(x_d-1)N+2^{d-2}(n_d-1-x_d)N \\
=S +2^{d-1}n_dM+2^{d-2}(n_d-2)N. 
\end{multline}

For the edges in $E_0(H)\cup E_1(H)$ we use the fact that each of the sets $V_\varepsilon(H)$
induces a subgraph isomorphic to $Q_{d-1}$. In addition, if $d$ is odd then exactly half of the
indices $i\in\{1,\ldots,d-1\}$ are even, and if $d$ is even then
\begin{itemize}
\item exactly half of the indices $i\in\{1,\ldots,d-2\}$ are even, and
\item for exactly half of the vertices $(y_1,\ldots,y_{d-2},x_{d-1},x_d+\varepsilon)$ the sum
  $y_1+\cdots y_{d-2}$ is even.
\end{itemize}
In both cases we conclude that for exactly half of the edges in $E_\varepsilon(H)$ the term added to
$\tilde g(\{(x_1,\ldots,x_{d-1}),(y_1,\ldots,y_{d-1})\})$ in~(\ref{eq:ConnectingEdges})
to~(\ref{eq:EdgeEvenCase2}) is $(x_d-1)M$, and for the other half it is $(n_d-x_d)M$. This implies
\begin{multline}\label{eq:horizontaledgesproof}
  \sum_{\{x,y\}\in E_\varepsilon(H)}g(\{x,y\})=\sum_{\{x,y\}\in E_\varepsilon(H)}\tilde
  g(\{(x_1,\ldots,x_{d-1}),(y_1,\ldots,y_{d-1})\})\\
  +(d-1)2^{d-3}(x_d-1)M+(d-1)2^{d-3}(n_d-x_d)M\\ = S'+(d-1)2^{d-3}(n_d-1)M.
\end{multline}
Combining~(\ref{eq:connectingedgesproof}) and~(\ref{eq:horizontaledgesproof}), the function $g$ is a
$Q_d$-magic labeling with $Q_d$-magic sum
\begin{multline*}
c'(g)=S+2^{d-1}n_dM+2\left(S'+(d-1)2^{d-3}(n_d-1)M\right)\\
=S+2S'+2^{d-2}(n_d-2)N+2^{d-2}\left(2n_d+(d-1)(n_d-1)\right)M.\qedhere	  
\end{multline*}
\end{proof}


\end{document}